\documentclass[a4paper,12pt]{article}
\usepackage{amsmath}
\usepackage{amssymb}
\usepackage{latexsym}
\usepackage{amsthm}
\usepackage[margin=1in]{geometry}
\usepackage{dsfont}
\usepackage{graphicx, color, caption, float, setspace, subfigure}
\usepackage{showlabels}

\theoremstyle{definition}
\newtheorem{theorem}{Theorem}[section]
\newtheorem{definition}{Definition}[section]
\newtheorem{lemma}{Lemma}[section]
\newtheorem{proposition}{Proposition}[section]
\newtheorem{remark}{Remark}[section]

\begin{document}
\begin{center}
\Large{\bf{On new types of fractional operators and applications}}
\end{center}
\vskip.6cm
\begin{center}
Mohamed Jleli,  Bessem Samet
\end{center}
\vskip.6cm

\begin{abstract}{\footnotesize
We introduce two kinds of fractional integral operators; the one is defined via the exponential-integral function 
$$
E_1(x)=\int_x^\infty \frac{e^{-t}}{t}\,dt,\quad x>0,
$$
and the other is defined via the special function
$$
\mathcal{S}(x)=e^{-x} \int_0^\infty \frac{x^{s-1}}{\Gamma(s)}\,ds,\quad x>0.
$$
We establish different properties of these operators, and we study the relationship between the fractional integrals of first kind and the fractional integrals of second kind. Next, we introduce a new concept of fractional derivative of order $\alpha>0$, which is defined via the fractional integral of first kind.  Using an approximate identity argument, we show that the introduced fractional derivative converges to the standard derivative in $L^1$ space, as $\alpha\to 0^+$. Several other properties are studied, like fractional integration by parts, the relationship between this fractional derivative and the fractional integral of second kind, etc.  As an application, we consider a new fractional model of the relaxation equation, we establish an existence and uniqueness result for this model, and provide an iterative algorithm that converges to the solution.\\ \\
\noindent {\bf 2010 Mathematics Subject Classification.} 26A33; 34A08.\\
{\bf Key words and phrases:} Fractional integrals of first kind; fractional integrals of second kind; fractional derivative; exponential-integral function; approximate identity; convergence; fractional relaxation equation.
}  
\end{abstract}

\section{Introduction}

In recent years, the theory of fractional calculus has become an important object of investigations due to its demonstrated applications  in different areas of physics and engineering, such as image segmentation \cite{GH}, fluid mechanics \cite{KU}, viscoelasticity \cite{BA,FA}, stochastic processes \cite{CO}, pollution phenomena \cite{JKS}, etc.

The classical fractional calculus is based  on the well-known Riemann-Liouville fractional integrals 
$$
(I_a^\alpha f)(x)=\frac{1}{\Gamma(\alpha)}\int_0^x(x-y)^{\alpha-1}f(y)\,dy,\quad \alpha>0
$$ 
and derivatives of order $\alpha$
$$
(D_0^\alpha f)(x)=\left(\frac{d}{dx}\right)^n (I_0^{n-\alpha} f)(x),
$$
where $\Gamma$ is the Gamma function and $n=[\alpha]+1$, or on the  Erd\'elyi-Kober operators as their immediate generalizations.  For more details on these fractional operators, see for example  \cite{KST,SKM}. There are also several other definitions of integrals and derivatives of fractional order, each of them depends on a particular kernel-function, see for example \cite{AG,AL,GA,HI,KIR,SR}.

In this paper, new fractional integral and derivative operators of order $\alpha>0$ are introduced. Different properties like the relationship between  fractional derivatives and  fractional integrals, convergence results as $\alpha\to 0^+$, fractional integration by parts, etc, are studied. Moreover, an application to a viscoelsaticity problem is provided.

The paper is organized as follows. In Section \ref{sec2}, we introduce the class of fractional integral operators of first kind $J_a^\alpha$ and $J_b^\alpha$, where $(a,b)\in \mathbb{R}^2$, $a<b$, and $\alpha>0$. The considered operators involve as a kernel-function the exponential-integral function 
$$
E_1(x)=\int_x^\infty \frac{e^{-t}}{t}\,dt,\quad x>0.
$$
We prove that 
$$
J_a^\alpha, J_b^\alpha: L^p([a,b];\mathbb{R})\to L^p([a,b];\mathbb{R}),
$$
where $1\leq p\leq \infty$, are linear and continuous operators (see Theorem \ref{T1} and Theorem \ref{T2}). Next, a fractional integration by parts result involving $J_a^\alpha$ and $J_b^\alpha$ operators is established (see Theorem \ref{T3}). Moreover, using an approximate identity argument, we prove that for any $f\in L^1([a,b];\mathbb{R})$, we have $J_a^\alpha f\to f$ and  $J_b^\alpha f\to f$ in $L^1([a,b];\mathbb{R})$, as $\alpha\to 0^+$ (see Theorem \ref{T4} and Theorem \ref{T5}). Further, we calculate $J_a^\alpha f$ and  $J_b^\alpha f$ for some particular functions $f$. In Section \ref{sec3}, we introduce the class of fractional integral operators of second kind $S_a^\alpha$ and $S_b^\alpha$, where $(a,b)\in \mathbb{R}^2$, $a<b$, and $\alpha>0$. These operators are defined via the special  function 
$$
\mathcal{S}(x)=e^{-x} \int_0^\infty \frac{x^{s-1}}{\Gamma(s)}\,ds,\quad x>0.
$$
We prove that 
$$
S_a^\alpha, S_b^\alpha: L^1([a,b];\mathbb{R})\to L^1([a,b];\mathbb{R})
$$
are linear and continuous operators (see Theorem \ref{TS1} and Theorem \ref{TS2}). As in the case of fractional integrals of first kind, an integration by parts result involving $S_a^\alpha$ and $S_b^\alpha$ operators is
derived (see Theorem \ref{TS3}). Next, we study the relationship between $J_a^\alpha$ and $S_a^\alpha $ operators, as well as the relationship between $J_b^\alpha$ and $S_b^\alpha $ operators (see Theorem \ref{THAYA} and Theorem \ref{THAYA2}). Moreover, we show that for any $f\in L^1([a,b];\mathbb{R})$, we have  $S_a^\alpha f\to \int_a^x f(y)\,dy$ and  $S_b^\alpha f\to \int_x^bf(y)\,dy$ in $L^1([a,b];\mathbb{R})$, as $\alpha\to 0^+$ (see Theorem \ref{TASS} and Theorem \ref{TASS2}). Further, we calculate 
 $S_a^\alpha f$ and  $S_b^\alpha f$, where $f$ is a constant function. In Section \ref{sec4}, we introduce a new class of fractional derivative operators $\mathcal{D}_a^\alpha$ and $\mathcal{D}_b^\alpha$, where $(a,b)\in \mathbb{R}^2$, $a<b$, and $\alpha>0$.  The derivative operator $\mathcal{D}_a^\alpha$ of a function $f$ satisfying a certain regularity, is defined as the first order derivative of $J_a^\alpha f$, while $\mathcal{D}_b^\alpha f$ is defined as the first order derivative of $J_b^\alpha f$. Under certain regularity conditions, we prove that $\mathcal{D}_a^\alpha f \to \frac{df}{dx}$ and $\mathcal{D}_b^\alpha f \to \frac{df}{dx}$ in $L^1([a,b];\mathbb{R})$, as $\alpha\to 0^+$ (see Theorem \ref{TAD}, Theorem \ref{TAD2} and Theorem \ref{fmf}). Next, we prove that $\mathcal{D}_a^\alpha$ is  a right invertible operator, whose inverse is $S_a^\alpha$ (see Theorem \ref{TINV1}). Similarly, we prove that $\mathcal{D}_b^\alpha$ is  a right invertible operator, whose inverse is $-S_b^\alpha$ (see Theorem \ref{TINV2}). Other relationships between $\mathcal{D}_a^\alpha$ and $\mathcal{S}_a^\alpha$ (resp. $\mathcal{D}_b^\alpha$ and $\mathcal{S}_b^\alpha$) are proved (see  Theorem \ref{TTK}, Theorem \ref{TTK2}, Theorem \ref{KATR} and Theorem \ref{KATR2}). In the case of an absolutely continuous function $f$, we obtain new representations of $\mathcal{D}_a^\alpha f$ and $\mathcal{D}_b^\alpha f$ (see Theorem \ref{TYYR} and Theorem \ref{ghghg}). Finally, an integration by parts result involving the fractional derivative operators $\mathcal{D}_a^\alpha$ and $\mathcal{D}_b^\alpha$ is derived (see Theorem \ref{RIP}). As an application, in Section \ref{sec5}, 
we consider a new fractional model of the relaxation equation, we establish an existence and uniqueness result for this model, and provide an iterative algorithm that converges to the solution.

\section{Fractional integral operators of first kind} \label{sec2}

In this section, we introduce a new class of fractional operators involving the exponential-integral function, and we study some of their properties.

The exponential-integral function is defined as (see, for example \cite{AS}):
$$
E_1(x)=\int_x^\infty \frac{e^{-t}}{t}\,dt,\quad  x>0.
$$
Using Fubini's theorem, it can be easily seen that 
\begin{equation}\label{ip}
\int_0^\infty E_1(x)\,dx=1.
\end{equation}

\begin{definition}[Left-sided fractional integral]
Let $f\in L^1([a,b];\mathbb{R})$, $(a,b)\in \mathbb{R}^2$, $a<b$,  be a given function. The left-sided fractional integral of order $\alpha>0$ of $f$ is given by
$$
\left(J_a^\alpha f\right)(x)=\frac{1}{\alpha} \int_a^x E_1\left(\frac{x-t}{\alpha}\right)f(t)\,dt,\quad \mbox{ a.e. } x\in [a,b].
$$ 
\end{definition}

Similarly, we define the right-sided fractional integral of order $\alpha>0$ as follows.

\begin{definition}[Right-sided fractional integral]
Let $f\in L^1([a,b];\mathbb{R})$, $(a,b)\in \mathbb{R}^2$, $a<b$,  be a given function. The right-sided fractional integral of order $\alpha>0$ of $f$ is given by
$$
\left(J_b^\alpha f\right)(x)=\frac{1}{\alpha} \int_x^b E_1\left(\frac{t-x}{\alpha}\right)f(t)\,dt,\quad \mbox{ a.e. } x\in [a,b].
$$ 
\end{definition}

We have the following result.

\begin{theorem}\label{T1}
Let $\alpha>0$, $1\leq p\leq \infty$ and $(a,b)\in \mathbb{R}^2$ be such that  $a<b$. Then
$$
J_a^\alpha: L^p([a,b];\mathbb{R})\to L^p([a,b];\mathbb{R})
$$
is a linear and continuous operator. Moreover, we have
$$
\left\|J_a^\alpha f\right\|_{ L^p([a,b];\mathbb{R})}\leq \|f\|_{ L^p([a,b];\mathbb{R})},\quad f\in  L^p([a,b];\mathbb{R}).
$$
\end{theorem}

\begin{proof}

Let $f\in L^p([a,b];\mathbb{R})$ be fixed. We define the function $F: \mathbb{R}\to \mathbb{R}$ by 
\begin{eqnarray}\label{F-func}
F(x)=\left\{ \begin{array}{lll}
f(x),  &&\mbox{a.e. } x\in [a,b],\\
0, &&\mbox{ otherwise.}
\end{array}
\right.
\end{eqnarray}
Similarly, we define the function $G_\alpha: \mathbb{R}\to \mathbb{R}$ by
\begin{eqnarray}\label{G-alpha}
G_\alpha(x)=\left\{ \begin{array}{lll}
\frac{1}{\alpha}E_1\left(\frac{x}{\alpha}\right),  &&\mbox{ if } x>0,\\
0, &&\mbox{ otherwise.}
\end{array}
\right.
\end{eqnarray}
Observe that $F\in L^p(\mathbb{R};\mathbb{R})$ and 
\begin{equation}\label{eq1}
\|F\|_{L^p(\mathbb{R};\mathbb{R})}=\|f\|_{L^p([a,b];\mathbb{R})}.
\end{equation}
Moreover, by \eqref{ip}, we have $G_\alpha\in L^1(\mathbb{R};\mathbb{R})$ and 
\begin{equation}\label{eq2}
\|G_\alpha\|_{L^1(\mathbb{R};\mathbb{R})}=1.
\end{equation}
Therefore, $F*G_\alpha\in L^p(\mathbb{R};\mathbb{R})$ and
$$
\|F*G_\alpha\|_{L^p(\mathbb{R};\mathbb{R})}\leq \|F\|_{L^p(\mathbb{R};\mathbb{R})} \|G_\alpha\|_{L^1(\mathbb{R};\mathbb{R})},
$$
where $*$ denotes the convolution product. Further, by \eqref{eq1} and \eqref{eq2}, we obtain
\begin{equation}\label{eq3}
\|F*G_\alpha\|_{L^p(\mathbb{R};\mathbb{R})}\leq  \|f\|_{L^p([a,b];\mathbb{R})}.
\end{equation}
On the other hand,  for a.e. $x\in [a,b]$, we have
\begin{eqnarray}\label{tist}
\nonumber (F*G_\alpha)(x)&=& \int_{\mathbb{R}} F(x-y)G_\alpha(y)\,dy\\
\nonumber &=& \int_{x-b}^{x-a} f(x-y) G_\alpha(y)\,dy\\
\nonumber &=& \frac{1}{\alpha}\int_{0}^{x-a} f(x-y) E_1\left(\frac{y}{\alpha}\right)\,dy\\
\nonumber &=& \frac{1}{\alpha}\int_{a}^{x} E_1\left(\frac{x-z}{\alpha}\right)f(z)\,dz\\
&=& \left(J_a^\alpha f\right)(x),
\end{eqnarray}
which yields
\begin{equation}\label{eq4}
\left\|J_a^\alpha f\right\|_{ L^p([a,b];\mathbb{R})}\leq \|F*G_\alpha\|_{L^p(\mathbb{R};\mathbb{R})}.
\end{equation}
Finally, combining \eqref{eq3} with \eqref{eq4}, the desired result follows.
\end{proof}

Using a similar argument to that above, we obtain the following result.

\begin{theorem}\label{T2}
Let $\alpha>0$, $1\leq p\leq \infty$ and $(a,b)\in \mathbb{R}^2$ be such that  $a<b$. Then
$$
J_b^\alpha: L^p([a,b];\mathbb{R})\to L^p([a,b];\mathbb{R})
$$
is a linear and continuous operator. Moreover, we have
$$
\left\|J_b^\alpha f\right\|_{ L^p([a,b];\mathbb{R})}\leq \|f\|_{ L^p([a,b];\mathbb{R})},\quad f\in  L^p([a,b];\mathbb{R}).
$$
\end{theorem}

Note that unlike the case of Riemann-Liouville fractional integrals, the  semigroup property is not satisfied in our case. More precisely, the following properties
$$
J_a^\alpha\left(J_a^\beta f\right)=J_a^{\alpha+\beta} f,\quad \alpha>0,\,\beta>0
$$
and
$$
J_b^\alpha\left(J_b^\beta f\right)=J_b^{\alpha+\beta} f,\quad \alpha>0,\,\beta>0
$$
do not hold in general.

We have the following integration by parts result.

\begin{theorem}\label{T3}
Let $1\leq p,q\leq \infty$ be such that $\frac{1}{p}+\frac{1}{q}=1$.  Given $\alpha>0$ and $(f,g)\in L^p([a,b];\mathbb{R})
\times L^q([a,b];\mathbb{R})$, $(a,b)\in \mathbb{R}^2$, $a<b$,  we have
$$
\int_a^b \left(J_a^\alpha f\right)(x)g(x)\,dx=
\int_a^b \left(J_b^\alpha g\right)(x)f(x)\,dx.
$$
\end{theorem}

\begin{proof}
First observe that by Theorem \ref{T1} and the regularity assumptions imposed on $f$ and $g$, we have
$$
\left|\int_a^b \left(J_a^\alpha f\right)(x)g(x)\,dx\right|\leq \|f\|_{L^p([a,b];\mathbb{R})} \|g\|_{ L^q([a,b];\mathbb{R})}<\infty.
$$
Next, using Fubini's theorem, we have
\begin{eqnarray*}
\int_a^b \left(J_a^\alpha f\right)(x)g(x)\,dx&=&
\int_a^b \left(\frac{1}{\alpha}\int_a^x E_1\left(\frac{x-t}{\alpha}\right)f(t)\,dt\right)g(x)\,dx\\
&=&  \int_a^b f(t)   \left(\int_t^b \frac{1}{\alpha}E_1\left(\frac{x-t}{\alpha}\right)g(x)\,dx\right)\,dt\\
&=& \int_a^b f(t) \left(J_b^\alpha g\right)(t)\,dt,
\end{eqnarray*}
which yields the desired result.
\end{proof}

Further, we shall prove the following approximation result.

\begin{theorem}\label{T4}
Let $f\in L^1([a,b];\mathbb{R})$, $(a,b)\in \mathbb{R}^2$, $a<b$. Then
$$
\lim_{\alpha\to 0^+} \left\|J_a^\alpha f-f\right\|_{L^1([a,b];\mathbb{R})}=0.
$$
\end{theorem}

Before giving the proof of Theorem \ref{T4}, we need some preliminaries on Harmonic Analysis (see, for example \cite{CH}).

\begin{definition}[approximate identity]
An approximate identity is a family
$\{h_\alpha\}_{\alpha>0}$ of real valued functions such that
\begin{itemize}
\item[\rm{(i)}] $h_\alpha\geq 0$, for all $\alpha>0$.
\item[\rm{(ii)}]$\int_{\mathbb{R}}h_\alpha(t)\,dt=1$, for all $\alpha>0$.
\item[\rm{(iii)}] For every $\delta>0$, 
$$
\lim_{\alpha\to 0^+} \int_{|x|>\delta} h_\alpha(t)\,dt =0.
$$
\end{itemize}
\end{definition}

\begin{lemma}\label{HAAA}
Let $h$ be a real valued function such that 
\begin{itemize}
\item[\rm{(i)}] $h\geq 0$.
\item[\rm{(ii)}]$\int_{\mathbb{R}}h(t)\,dt=1$.
\end{itemize}
Then, the family $\{h_\alpha\}_{\alpha>0}$ given by
$$
h_\alpha(t)=\frac{1}{\alpha}h\left(\frac{t}{\alpha}\right),\quad t\in \mathbb{R},\,\alpha>0
$$
is an approximate identity. 
\end{lemma}

\begin{lemma}\label{HAAA2}
Let $\{h_\alpha\}_{\alpha>0}$ be an approximate identity. If $F\in L^1(\mathbb{R};\mathbb{R})$, then
$$
F*h_\alpha\in L^1(\mathbb{R};\mathbb{R}),\quad \alpha>0.
$$
Moreover, we have
$$
\lim_{\alpha\to 0^+} \|F*h_\alpha-F\|_{L^1(\mathbb{R};\mathbb{R})}=0.
$$
\end{lemma}

Let $\widetilde{E_1}: \mathbb{R}\to \mathbb{R}$ be the function given by
\begin{eqnarray*}
\widetilde{E_1}(x)=\left\{ \begin{array}{lll}
E_1(x),  &&\mbox{if } x>0,\\
0, &&\mbox{ otherwise.}
\end{array}
\right.
\end{eqnarray*}
Observe that by \eqref{ip}, the function $\widetilde{E_1}$ satisfies the assumptions of Lemma \ref{HAAA}. Therefore, the function $\{G_\alpha\}_{\alpha>0}$ given by \eqref{G-alpha} is an approximate identity. Hence, by Lemma \ref{HAAA2}, we have the following result.

\begin{lemma}\label{HAAA3}
For every $F\in L^1(\mathbb{R};\mathbb{R})$, we have
$$
\lim_{\alpha\to 0^+} \|F*G_\alpha-F\|_{L^1(\mathbb{R};\mathbb{R})}=0.
$$
\end{lemma}

Now, we are ready to give the proof of Theorem \ref{T4}.

\begin{proof}
Let $f\in L^1([a,b];\mathbb{R})$. Then $F\in L^1(\mathbb{R};\mathbb{R})$, where $F$ is the function given by \eqref{F-func}. By Lemma \ref{HAAA3}, we have
$$
\lim_{\alpha\to 0^+} \|F*G_\alpha-F\|_{L^1(\mathbb{R};\mathbb{R})}=0,
$$
which yields
\begin{equation}\label{aya}
\lim_{\alpha\to 0^+} \|(F*G_\alpha)_{|[a,b]}-f\|_{L^1([a,b];\mathbb{R})}=0.
\end{equation}
On the other hand, by \eqref{tist}, we have
\begin{equation}\label{aya2}
(F*G_\alpha)_{|[a,b]}=J_a^\alpha f \,\mbox{ in }\, L^1([a,b];\mathbb{R}).
\end{equation}
Combining \eqref{aya} with \eqref{aya2}, the desired result follows.
\end{proof}

Using a similar argument to that above, we obtain  the following approximation result for $J_b^\alpha f$.

\begin{theorem}\label{T5}
Let $f\in L^1([a,b];\mathbb{R})$, $(a,b)\in \mathbb{R}^2$, $a<b$. Then
$$
\lim_{\alpha\to 0^+} \left\|J_b^\alpha f-f\right\|_{L^1([a,b];\mathbb{R})}=0.
$$
\end{theorem}

Next, we give some examples of functions $f$ for which we calculate $J_a^\alpha f$ and $J_b^\alpha f$.

\begin{proposition}\label{ex1}
Let $(a,b)\in \mathbb{R}^2$ be such that $a<b$. Let $f\equiv C$, where $C$ is a real constant function. Given $\alpha>0$, we have 
\begin{equation}\label{cas1}
\left(J_a^\alpha f\right)(x)=C\left[\left(\frac{x-a}{\alpha} \right)E_1\left(\frac{x-a}{\alpha}\right)-e^{\frac{-(x-a)}{\alpha}}+1\right],\quad \mbox{ a.e. } x\in [a,b]
\end{equation}
and
\begin{equation}\label{cas2}
\left(J_b^\alpha f\right)(x)=C\left[\left(\frac{b-x}{\alpha} \right)E_1\left(\frac{b-x}{\alpha}\right)-e^{\frac{-(b-x)}{\alpha}}+1\right],\quad \mbox{ a.e. } x\in [a,b].
\end{equation}
\end{proposition}

\begin{proof}
\eqref{cas1} and \eqref{cas2} follow immediately using the properties (see \cite{G}):
$$
\int E_1(x)\,dx= x E_1(x)-e^{-x}
$$
and 
\begin{equation}\label{pp2}
\lim_{x\to 0^+} xE_1(x)=0.
\end{equation}
\end{proof}

\begin{proposition}\label{ex2}
Let $(a,b)\in \mathbb{R}^2$ be such that $a<b$. Let $f_n: [a,b]\to \mathbb{R}$, $n\in \mathbb{N}\cup\{0\}$,   be the function given by
$$
f_n(x)=x^n,\quad x\in [a,b].
$$
Given $\alpha>0$, for a.e. $x\in [a,b]$, we have 
\begin{align}\label{cas21}
\notag &\left(J_a^\alpha f_n\right)(x)=\\
& \sum_{k=0}^n (-\alpha)^kC_n^k x^{n-k} \left[\frac{1}{k+1} \left(\frac{x-a}{\alpha}\right)^{k+1} E_1\left(\frac{x-a}{\alpha}\right)-\frac{k!}{k+1} e_k\left(\frac{x-a}{\alpha}\right)e^{\frac{-(x-a)}{\alpha}}+1\right]
\end{align}
and
\begin{align}\label{cas22}
\notag &\left(J_b^\alpha f_n\right)(x)=\\
& \sum_{k=0}^n \alpha^k C_n^k x^{n-k} \left[\frac{1}{k+1} \left(\frac{b-x}{\alpha}\right)^{k+1} E_1\left(\frac{b-x}{\alpha}\right)-\frac{k!}{k+1} e_k\left(\frac{b-x}{\alpha}\right)e^{\frac{-(b-x)}{\alpha}}+1\right],
\end{align}
where 
$$
e_k(x)=\sum_{i=0}^k \frac{x^i}{i!}.
$$
\end{proposition}

\begin{proof}
\eqref{cas21} and \eqref{cas22} follow immediately using \eqref{pp2} and the property (see \cite{G}):
$$
\int x^n E_1(x)\,dx= \frac{x^{n+1}}{n+1}E_1(x)-\frac{n!}{n+1}e_n(x)e^{-x},\quad n\in \mathbb{N}\cup \{0\}.
$$
\end{proof}

\begin{remark}
Taking $n=0$, \eqref{cas21} reduces to \eqref{cas1}  and \eqref{cas22} reduces to \eqref{cas2} (with $C=1$).
\end{remark}

Using a similar calculation to that above, we obtain the following result.

\begin{proposition}\label{ex3}
Let $(a,b)\in \mathbb{R}^2$ be such that $a<b$. Let $f_n: [a,b]\to \mathbb{R}$, $n\in \mathbb{N}\cup\{0\}$,   be the function given by
$$
f_n(x)=(x-a)^n,\quad x\in [a,b].
$$
Let $g_n: [a,b]\to \mathbb{R}$, $n\in \mathbb{N}\cup\{0\}$,   be the function given by
$$
g_n(x)=(b-x)^n,\quad x\in [a,b].
$$
Given $\alpha>0$, for a.e. $x\in [a,b]$, we have 
\begin{align*}
&\left(J_a^\alpha f_n\right)(x)=\\
& \sum_{k=0}^n (-\alpha)^kC_n^k (x-a)^{n-k} \left[\frac{1}{k+1} \left(\frac{x-a}{\alpha}\right)^{k+1} E_1\left(\frac{x-a}{\alpha}\right)-\frac{k!}{k+1} e_k\left(\frac{x-a}{\alpha}\right)e^{\frac{-(x-a)}{\alpha}}+1\right]
\end{align*}
and
\begin{align*}
 &\left(J_b^\alpha g_n\right)(x)=\\
& \sum_{k=0}^n (-\alpha)^k C_n^k (b-x)^{n-k} \left[\frac{1}{k+1} \left(\frac{b-x}{\alpha}\right)^{k+1} E_1\left(\frac{b-x}{\alpha}\right)-\frac{k!}{k+1} e_k\left(\frac{b-x}{\alpha}\right)e^{\frac{-(b-x)}{\alpha}}+1\right].
\end{align*}
\end{proposition}

\begin{proposition}\label{ex4}
Let $(a,b)\in \mathbb{R}^2$ be such that $a<b$. Let $f_\alpha: [a,b]\to \mathbb{R}$, $\alpha>0$,    be the function given by
$$
f_\alpha(x)=E_1\left(\frac{x-a}{\alpha}\right),\quad a<x\leq b.
$$
Then, for a.e. $x\in [a,b]$, we have 
\begin{align*}
&\left(J_a^\alpha f_\alpha\right)(x)=\\
&2\left[\gamma+\ln\left(\frac{x-a}{\alpha}\right)\right]e^{\frac{-(x-a)}{\alpha}}+2\left[1-\left(\frac{x-a}{\alpha}\right)\gamma-\left(\frac{x-a}{\alpha}\right)\ln\left(\frac{x-a}{\alpha}\right)\right]E_1\left(\frac{x-a}{\alpha}\right)\\
&-\left(\frac{x-a}{\alpha}\right) \left[\zeta(2)+\left(\gamma+\ln\left(\frac{x-a}{\alpha}\right)\right)^2\right]-2\left(\frac{x-a}{\alpha}\right)\sum_{m=1}^\infty \frac{(-1)^m}{m!m^2}\left(\frac{x-a}{\alpha}\right)^m,
\end{align*}
where $\gamma$ is the Euler's constant and $\zeta$ is  the Euler-Riemann zeta function.
\end{proposition}

\begin{proof}
The result follows using the property (see \cite{G}):
\begin{align}\label{pj}
\notag &\int_0^r E_1(x)E_1(r-x)\,dx\\
\notag &= 2(\gamma+\ln r)e^{-r}+2\left(1-\gamma r-r\ln r\right)E_1(r)-r\left(\zeta(2)+(\gamma+\ln r)^2\right)\\
&- 2 r \sum_{m=1}^\infty \frac{(-r)^m}{m!m^2},
\end{align}
where $r>0$.
\end{proof}

Similarly, using \eqref{pj}, we obtain the following result.

\begin{proposition}\label{ex5}
Let $(a,b)\in \mathbb{R}^2$ be such that $a<b$. Let $g_\alpha: [a,b]\to \mathbb{R}$, $\alpha>0$,    be the function given by
$$
g_\alpha(x)=E_1\left(\frac{b-x}{\alpha}\right),\quad a\leq x< b.
$$
Then, for a.e. $x\in [a,b]$, we have 
\begin{align*}
&\left(J_b^\alpha g_\alpha\right)(x)=\\
&2\left[\gamma+\ln\left(\frac{b-x}{\alpha}\right)\right]e^{\frac{-(b-x)}{\alpha}}+2\left[1-\left(\frac{b-x}{\alpha}\right)\gamma-\left(\frac{b-x}{\alpha}\right)\ln\left(\frac{b-x}{\alpha}\right)\right]E_1\left(\frac{b-x}{\alpha}\right)\\
&-\left(\frac{b-x}{\alpha}\right) \left[\zeta(2)+\left(\gamma+\ln\left(\frac{b-x}{\alpha}\right)\right)^2\right]-2\left(\frac{b-x}{\alpha}\right)\sum_{m=1}^\infty \frac{(-1)^m}{m!m^2}\left(\frac{b-x}{\alpha}\right)^m.
\end{align*}
\end{proposition}

\section{Fractional integral operators of second kind} \label{sec3}

In this section, we present another class of fractional integral operators. First,  let us introduce some special functions that will be used later (see, for example \cite{AS}).

We introduce the function
$$
\mathcal{S}(t)=e^{-t} \int_0^\infty \frac{t^{s-1}}{\Gamma(s)}\,ds, \quad t>0.
$$ 
We denote by $\gamma(\cdot,\cdot)$  the  lower incomplete Gamma function, given by
$$
\gamma(s,x)=\int_0^xt^{s-1}e^{-t}\,dt,\quad s>0.\,x>0.
$$
The regularized lower Gamma function is given by
$$
P(s,x)=\frac{\gamma(s,x)}{\Gamma(s)},\quad s>0,\,x>0.
$$
The derivative of $P(s,x)$ with respect to $x$ is given by
\begin{equation}\label{dig}
\frac{d}{dx}P(s,x)=\frac{x^{s-1} e^{-x}}{\Gamma(s)},\quad s>0,\, x>0.
\end{equation}

\begin{definition}[Left-sided fractional integral]
Let $f\in L^1([a,b];\mathbb{R})$, $(a,b)\in \mathbb{R}^2$, $a<b$,  be a given function. The left-sided fractional integral of order $\alpha>0$ of $f$ is given by
$$
\left(S_a^\alpha f\right)(x)= \int_a^x \mathcal{S}\left(\frac{x-t}{\alpha}\right)f(t)\,dt,\quad \mbox{ a.e. } x\in [a,b].
$$ 
\end{definition}

Similarly, we define the right-sided fractional integral of order $\alpha>0$ as follows.

\begin{definition}[Right-sided fractional integral]
Let $f\in L^1([a,b];\mathbb{R})$, $(a,b)\in \mathbb{R}^2$, $a<b$,  be a given function. The right-sided fractional integral of order $\alpha>0$ of $f$ is given by
$$
\left(S_b^\alpha f\right)(x)= \int_x^b \mathcal{S}\left(\frac{t-x}{\alpha}\right)f(t)\,dt,\quad \mbox{ a.e. } x\in [a,b].
$$ 
\end{definition}

We have the following result.

\begin{theorem}\label{TS1}
Let $\alpha>0$  and $(a,b)\in \mathbb{R}^2$ be such that  $a<b$. Then
$$
S_a^\alpha: L^1([a,b];\mathbb{R})\to L^1([a,b];\mathbb{R})
$$
is a linear and continuous operator. Moreover, we have
$$
\left\|S_a^\alpha f\right\|_{ L^1([a,b];\mathbb{R})}\leq \alpha \left(\int_0^\infty P\left(s,\frac{b-a}{\alpha}\right)\,ds\right)\|f\|_{ L^1([a,b];\mathbb{R})},\quad f\in  L^1([a,b];\mathbb{R}).
$$
\end{theorem}

\begin{proof}
Let $f\in  L^1([a,b];\mathbb{R})$ be fixed. Using Fubini's theorem, we have
\begin{eqnarray}\label{cup}
\nonumber \left\|S_a^\alpha f\right\|_{ L^1([a,b];\mathbb{R})} &= & 
\int_a^b \left|\left(S_a^\alpha f\right)(x)\right|\,dx\\
\nonumber &\leq & \int_a^b \left(\int_a^x \mathcal{S}\left(\frac{x-t}{\alpha}\right)|f(t)|\,dt\right)\,dx\\
&=& \int_a^b |f(t)| \left(\int_t^b \mathcal{S}\left(\frac{x-t}{\alpha}\right)\,dx\right)\,dt.
\end{eqnarray}
On the other hand, for a.e. $t\in [a,b]$, using again Fubini's theorem, we have
\begin{eqnarray}\label{cap}
\nonumber  \int_t^b \mathcal{S}\left(\frac{x-t}{\alpha}\right)\,dx&=& \int_t^b \left(e^{\frac{-(x-t)}{\alpha}}\int_0^\infty \frac{\left(\frac{x-t}{\alpha}\right)^{s-1}}{\Gamma(s)}\,ds\right)\,dx\\
&=& \int_0^\infty \frac{1}{\Gamma(s)} \left(\int_t^b \left(\frac{x-t}{\alpha}\right)^{s-1} e^{\frac{-(x-t)}{\alpha}}\,dx\right)\,ds.
\end{eqnarray}
Using the change of variable
$$
y=\frac{x-t}{\alpha},
$$
we obtain
\begin{eqnarray}\label{nice}
\nonumber \int_t^b \left(\frac{x-t}{\alpha}\right)^{s-1} e^{\frac{-(x-t)}{\alpha}}\,dx&=&\alpha \int_0^{\frac{b-t}{\alpha}} y^{s-1} e^{-y} \,dy\\
\nonumber &=&\alpha \gamma\left(s,\frac{b-t}{\alpha}\right)\\
&\leq & \alpha \gamma\left(s,\frac{b-a}{\alpha}\right).
\end{eqnarray}
Next, combining \eqref{cap} with \eqref{nice}, we get
\begin{equation}\label{chnoua}
\int_t^b \mathcal{S}\left(\frac{x-t}{\alpha}\right)\,dx\leq 
\alpha \int_0^\infty P\left(s,\frac{b-a}{\alpha}\right)\,ds.
\end{equation}
Finally, combining \eqref{cup} with \eqref{chnoua}, we get the desired result.
\end{proof}

Using a similar argument to that above, we obtain the following result.

\begin{theorem}\label{TS2}
Let $\alpha>0$  and $(a,b)\in \mathbb{R}^2$ be such that  $a<b$. Then
$$
S_b^\alpha: L^1([a,b];\mathbb{R})\to L^1([a,b];\mathbb{R})
$$
is a linear and continuous operator. Moreover, we have
$$
\left\|S_b^\alpha f\right\|_{ L^1([a,b];\mathbb{R})}\leq \alpha \left(\int_0^\infty P\left(s,\frac{b-a}{\alpha}\right)\,ds\right) \|f\|_{ L^1([a,b];\mathbb{R})},\quad f\in  L^1([a,b];\mathbb{R}).
$$
\end{theorem}

Note that, as in the case of fractional integral operators of first kind, the following properties
$$
S_a^\alpha\left(S_a^\beta f\right)=S_a^{\alpha+\beta} f,\quad \alpha>0,\,\beta>0
$$
and
$$
S_b^\alpha\left(S_b^\beta f\right)=S_b^{\alpha+\beta} f,\quad \alpha>0,\,\beta>0
$$
do not hold in general.

Further, using a similar argument to that used in the proof of Theorem \ref{T3}, we obtain the following integration by parts result.

\begin{theorem}\label{TS3}
Given $\alpha>0$ and $(f,g)\in L^1([a,b];\mathbb{R})
\times L^\infty([a,b];\mathbb{R})$, $(a,b)\in \mathbb{R}^2$, $a<b$,  we have
$$
\int_a^b \left(S_a^\alpha f\right)(x)g(x)\,dx=
\int_a^b \left(S_b^\alpha g\right)(x)f(x)\,dx.
$$
\end{theorem}

Next, we shall study  the relation between the fractional integral operators of first and second kind. First, we need the following result, which can be easily proved.

\begin{lemma}\label{Laplace}
We have
$$
\left(\mathcal{L}E_1\right)(\lambda)=\frac{\ln(1+\lambda)}{\lambda},\quad \lambda>0
$$
and
$$
\left(\mathcal{L}\mathcal{S}\right)(\lambda)=\frac{1}{\ln(1+\lambda)},\quad \lambda>0,
$$
where $\mathcal{L}$ is the Laplace transform operator.
\end{lemma}

We have the following result

\begin{theorem}\label{THAYA}
Let $\alpha>0$  and $(a,b)\in \mathbb{R}^2$ be such that  $a<b$.  Then, for every $f\in L^1([a,b];\mathbb{R})$, we have
$$
J_a^\alpha \left(S_a^\alpha f\right)(x)=S_a^\alpha \left(J_a^\alpha f\right)(x)=\int_a^x f(y)\,dy,\quad \mbox{ a.e. } x\in [a,b].
$$
\end{theorem}

\begin{proof}
Let $f\in L^1([a,b];\mathbb{R})$ be fixed. Using Fubini's theorem, for a.e. $x\in [a,b]$, we have
\begin{eqnarray}\label{good}
\nonumber J_a^\alpha \left(S_a^\alpha f\right)(x)&=&\frac{1}{\alpha} \int_a^x E_1\left(\frac{x-t}{\alpha}\right) \left(S_a^\alpha f\right)(t)\,dt\\
\nonumber &=& \frac{1}{\alpha} \int_a^x E_1\left(\frac{x-t}{\alpha}\right) \left(\int_a^t \mathcal{S}\left(\frac{t-y}{\alpha}\right)f(y)\,dy\right)\,dt\\
&=&\int_a^x f(y) \left(\int_y^x \frac{1}{\alpha} E_1\left(\frac{x-t}{\alpha}\right) \mathcal{S}\left(\frac{t-y}{\alpha}\right) \,dt\right)\,dy.
\end{eqnarray}
Using the change of variable
$$
z=\frac{t-y}{\alpha},\quad a<y<x,
$$
we obtain
$$
\int_y^x \frac{1}{\alpha} E_1\left(\frac{x-t}{\alpha}\right) \mathcal{S}\left(\frac{t-y}{\alpha}\right) \,dt=\int_0^{\frac{x-y}{\alpha}} E_1\left(\frac{x-y}{\alpha}-z\right)\mathcal{S}(z)\,dz,
$$
i.e,
\begin{equation}\label{ouf}
\int_y^x \frac{1}{\alpha} E_1\left(\frac{x-t}{\alpha}\right) \mathcal{S}\left(\frac{t-y}{\alpha}\right) \,dt=\left(E_1*\mathcal{S}\right)\left(\frac{x-y}{\alpha}\right).
\end{equation}
On the other hand, using Lemma \ref{Laplace}, we have
$$
\mathcal{L} \left(E_1*\mathcal{S}\right)(\lambda)=\mathcal{L}(E_1)(\lambda)\, \mathcal{L}(\mathcal{S})(\lambda)=\frac{1}{\lambda}=\mathcal{L}(1)(\lambda),\quad \lambda>0,
$$
which yields
\begin{equation}\label{iz}
\left(E_1*\mathcal{S}\right)(\lambda)=1,\quad \lambda>0.
\end{equation}
Finally, combining \eqref{good}, \eqref{ouf} and \eqref{iz}, we obtain
$$
J_a^\alpha \left(S_a^\alpha f\right)(x)=\int_a^x f(y)\,dy,\quad \mbox{ a.e. } x\in [a,b].
$$
Note that due to the symmetry of the convolution product, it can be easily seen that 
$$
J_a^\alpha \left(S_a^\alpha f\right)(x)=S_a^\alpha \left(J_a^\alpha f\right)(x),\quad \mbox{ a.e. } x\in [a,b].
$$
This makes end to the proof.
\end{proof}

Using a similar argument to that above, we obtain the following result.

\begin{theorem}\label{THAYA2}
Let $\alpha>0$  and $(a,b)\in \mathbb{R}^2$ be such that  $a<b$.  Then, for every $f\in L^1([a,b];\mathbb{R})$, we have
$$
J_b^\alpha \left(S_b^\alpha f\right)(x)=S_b^\alpha \left(J_b^\alpha f\right)(x)=\int_x^b f(y)\,dy,\quad \mbox{ a.e. } x\in [a,b].
$$
\end{theorem}

Further, we shall prove the following approximation result.

\begin{theorem}\label{TASS}
Let $f\in L^1([a,b];\mathbb{R})$, $(a,b)\in \mathbb{R}^2$, $a<b$. Then
$$
\lim_{\alpha\to 0^+}\left\|S_a^\alpha f- I_af\right\|_{L^1([a,b];\mathbb{R})}=0,
$$
where
$$
\left(I_af\right)(x)=\int_a^x f(y)\,dy,\quad a<x\leq b.
$$
\end{theorem}

In order to prove Theorem \ref{TASS}, we need the following auxiliary  result.

\begin{lemma}\label{aux}
Given $\alpha_0>0$, we have
$$
\sup_{0<\alpha<\alpha_0} \alpha \int_0^{\frac{1}{\alpha}}\mathcal{S}(y)\,dy<\infty.
$$
\end{lemma}

\begin{proof}
First, by \eqref{ip}, there exists a certain $\delta>0$ such that  
\begin{equation}\label{ijdida}
\mu(z)>\frac{1}{2},\quad z\geq \delta,
\end{equation}
where
$$
\mu(z)=\int_0^zE_1(y)\,dy,\quad z>0.
$$
Next, let $0<\alpha<\alpha_0$ be fixed. Taking $\beta>0$ such that 
$$
\frac{1}{\beta}=\frac{1}{\alpha}+\delta
$$
and using \eqref{iz}, we obtain
$$
\int_0^{\frac{x}{\beta}} E_1\left(\frac{x}{\beta}-t\right)\mathcal{S}(t)\,dt=1,\quad x>0.
$$
Integrating over $[0,1]$, using Fubini's theorem and the change of variable 
$$
y=\frac{x}{\beta}-t,
$$
we obtain
$$
\beta \int_0^{\frac{1}{\beta}} \mathcal{S}(t)\mu\left(\frac{1}{\beta}-t\right)\,dt=1,
$$
i.e.,
$$
\beta\int_0^{\frac{1}{\beta}} \mathcal{S}\left(\frac{1}{\beta}-t\right)\mu(t)\,dt=1,
$$
which yields
$$
\beta\int_\delta^{\frac{1}{\beta}} \mathcal{S}\left(\frac{1}{\beta}-t\right)\mu(t)\,dt\leq 1.
$$
Using the above inequality and \eqref{ijdida}, we get
$$
\frac{\beta}{2} \int_\delta^{\frac{1}{\beta}} \mathcal{S}\left(\frac{1}{\beta}-t\right)\,dt\leq 1.
$$
Next, the change of variable
$$
y=\frac{1}{\beta}-t
$$
leads to 
$$
\frac{\beta}{2} \int_0^{\frac{1}{\beta}-\delta}\mathcal{S}(y)\,dy\leq 1,
$$
i.e.,
$$
\alpha\int_0^{\frac{1}{\alpha}} \mathcal{S}(y)\,dy \leq 2(1+\alpha \delta)\leq 2(1+\alpha_0\delta),
$$
which proves the desired result.
\end{proof}

Now, we are ready to give the proof of Theorem \ref{TASS}.

\begin{proof}
Let $f\in L^1([a,b];\mathbb{R})$ and  $\alpha>0$ be small enough. Using Theorem \ref{THAYA},  we have
\begin{eqnarray}\label{mah}
\nonumber \left\|S_a^\alpha f- I_af\right\|_{L^1([a,b];\mathbb{R})}&=&\left\|S_a^\alpha f- S_a^\alpha\left(J_a^\alpha f\right)\right\|_{L^1([a,b];\mathbb{R})}\\
&=& \left\|S_a^\alpha\left(J_a^\alpha f-f\right)\right\|_{L^1([a,b];\mathbb{R})}.
\end{eqnarray}
On the other hand, by Theorem \ref{TS1},  we have
\begin{equation}\label{poy}
\left\|S_a^\alpha\left(J_a^\alpha f-f\right)\right\|_{L^1([a,b];\mathbb{R})}\leq \alpha \left(\int_0^\infty P\left(s,\frac{b-a}{\alpha}\right)\,ds\right)\left\|J_a^\alpha f-f\right\|_{ L^1([a,b];\mathbb{R})}.
\end{equation}
Using Fubini's theorem, it can be easily seen that 
$$
\int_0^\infty  P\left(s,\frac{b-a}{\alpha}\right)\,ds=\int_0^{\frac{b-a}{\alpha}} \mathcal{S}(t)\,dt.
$$
Therefore, from \eqref{poy}, we get
$$
\left\|S_a^\alpha\left(J_a^\alpha f-f\right)\right\|_{L^1([a,b];\mathbb{R})}\leq \alpha \left(\int_0^{\frac{b-a}{\alpha}} \mathcal{S}(t)\,dt\right) \left\|J_a^\alpha f-f\right\|_{ L^1([a,b];\mathbb{R})}.
$$
Next, by Lemma \ref{aux}, there exists a certain constant $C>0$ (independent on $\alpha$) such that 
\begin{equation}\label{pyyg}
\left\|S_a^\alpha\left(J_a^\alpha f-f\right)\right\|_{L^1([a,b];\mathbb{R})}\leq C \left\|J_a^\alpha f-f\right\|_{ L^1([a,b];\mathbb{R})}.
\end{equation}
Combining \eqref{mah} with \eqref{pyyg}, we get
$$
\left\|S_a^\alpha f- I_af\right\|_{L^1([a,b];\mathbb{R})}\leq 
C \left\|J_a^\alpha f-f\right\|_{ L^1([a,b];\mathbb{R})}.
$$
Finally, passing to the limit as $\alpha\to 0^+$, and using Theorem \ref{T4}, the desired result follows.
\end{proof}

Using a similar argument to that above, we obtain the following approximation result for $S_b^\alpha f$.

\begin{theorem}\label{TASS2}
Let $f\in L^1([a,b];\mathbb{R})$, $(a,b)\in \mathbb{R}^2$, $a<b$. Then
$$
\lim_{\alpha\to 0^+}\left\|S_b^\alpha f- I_bf\right\|_{L^1([a,b];\mathbb{R})}=0,
$$
where
$$
\left(I_bf\right)(x)=\int_x^b f(y)\,dy,\quad a\leq x< b.
$$
\end{theorem}

In the case of a constant function $f$,  $S_a^\alpha f$ and $S_b^\alpha f$ are given by the following result.

\begin{proposition}\label{exS1}
Let $(a,b)\in \mathbb{R}^2$ be such that $a<b$. Let $f\equiv C$, where $C$ is a real constant function. Given $\alpha>0$, we have 
\begin{equation}\label{cas1S}
\left(S_a^\alpha f\right)(x)=\alpha C \int_0^\infty P\left(s,\frac{x-a}{\alpha}\right)\,ds ,\quad \mbox{ a.e. } x\in [a,b]
\end{equation}
and
\begin{equation}\label{cas2S}
\left(S_b^\alpha f\right)(x)=\alpha C \int_0^\infty P\left(s,\frac{b-x}{\alpha}\right)\,ds,\quad \mbox{ a.e. } x\in [a,b].
\end{equation}
\end{proposition}

\begin{proof}
\eqref{cas1S} and \eqref{cas2S} follow immediately using Fubini's theorem.
\end{proof}

\section{Fractional derivative operators}\label{sec4}

In this section, we introduce a new class of fractional derivative operators involving the fractional integrals of first kind. As we will show later, these new fractional derivatives are inverse operations  for the fractional integrals of second kind.

First, given $(a,b)\in \mathbb{R}^2$, $a<b$, we denote by 
$AC([a,b];\mathbb{R})$ the space of real valued and absolutely continuous functions in $[a,b]$. Recall that $f\in AC([a,b];\mathbb{R})$ if and only if there exist $(c,\varphi)\in \mathbb{R}\times L^1([a,b];\mathbb{R})$ such that 
$$
f(x)=c+\int_a^x \varphi(t)\,dt,\quad x\in [a,b].
$$
In this case, we have
$$
c=f(a)
$$
and
$$
\frac{df}{dx}(x)=\varphi(x),\quad \mbox{ a.e. } x\in [a,b].
$$

\begin{definition}
Let $f\in L^1([a,b];\mathbb{R})$, $(a,b)\in \mathbb{R}^2$, $a<b$. We say that $f$ has  an absolutely continuous representative, if there exists a function $\psi\in AC([a,b];\mathbb{R})$ such that 
$$
f(x)=\psi(x),\quad \mbox{ a.e. } x\in [a,b].
$$
In this case, the function $f$ is identified to its absolutely continuous representative $\psi$.
\end{definition}

\begin{definition}[Left-sided fractional derivative]
Let $f\in L^1([a,b];\mathbb{R})$, $(a,b)\in \mathbb{R}^2$, $a<b$,  be a given function such that $J_a^\alpha f$, $\alpha>0$, has  an absolutely continuous representative. 
The left-sided fractional derivative of order $\alpha$ of $f$ is given by
$$
\left(\mathcal{D}_a^\alpha f\right)(x)=\frac{d}{dx} \left(J_a^\alpha f\right)(x),\quad \mbox{ a.e. } x\in [a,b].
$$ 
\end{definition}

Similarly, we define the right-sided fractional derivative of order $\alpha>0$ as follows.

\begin{definition}[Right-sided fractional derivative]
Let $f\in L^1([a,b];\mathbb{R})$, $(a,b)\in \mathbb{R}^2$, $a<b$,  be a given function such that $J_b^\alpha f$, $\alpha>0$, has  an absolutely continuous representative. The right-sided fractional derivative of order $\alpha$ of $f$ is given by
$$
\left(\mathcal{D}_b^\alpha f\right)(x)=\frac{d}{dx} \left(J_b^\alpha f\right)(x),\quad \mbox{ a.e. } x\in [a,b].
$$ 
\end{definition}

Next, let us consider the functional space $V([a,b];\mathbb{R})$, $(a,b)\in \mathbb{R}^2$, $a<b$, given by 
$$
V([a,b];\mathbb{R})=\left\{u\in C^1([a,b];\mathbb{R}):\, u(a)=\frac{du}{dx}(a)=0,\, u(b)=\frac{du}{dx}(b)=0\right\}.
$$
We have the following approximation result.

\begin{theorem}\label{TAD}
Given $f\in V([a,b];\mathbb{R})$, $(a,b)\in \mathbb{R}^2$, $a<b$, we have
$$
\lim_{\alpha\to 0^+}\left\|\mathcal{D}_a^\alpha f -\frac{df}{dx}\right\|_{L^1([a,b];\mathbb{R})}=0.
$$
\end{theorem}

Before giving the proof of Theorem \ref{TAD}, we need the following standard result form Harmonic Analysis (see, for example \cite{CH}).

\begin{lemma}\label{finj}
Let $\mu\in L^1(\mathbb{R};\mathbb{R})$ and $\nu \in C_b^1(\mathbb{R};\mathbb{R})$, where $C_b^1(\mathbb{R};\mathbb{R})$ denotes the set of functions $u\in C^1(\mathbb{R};\mathbb{R})$ such that $u$ and $\frac{du}{dx}$ are bounded. Then  $\mu*\nu \in C_b^1(\mathbb{R};\mathbb{R})$ and 
$$
\frac{d(\mu*\nu)}{dx}(x)=\frac{d(\nu*\mu)}{dx}(x)=\left(\mu*\frac{d\nu}{dx}\right)(x),\quad x\in \mathbb{R}.
$$
\end{lemma}

Now, we are ready to give the proof of Theorem \ref{TAD}.

\begin{proof}
Let $f\in V([a,b];\mathbb{R})$ be fixed. We introduce the function $F: \mathbb{R}\to \mathbb{R}$ given by
\begin{eqnarray*}
F(x)=\left\{ \begin{array}{lll}
f(x),  &&\mbox{if } x\in [a,b],\\
0, &&\mbox{ otherwise.}
\end{array}
\right.
\end{eqnarray*}
It can be easily seen that $F\in C_b^1(\mathbb{R};\mathbb{R})$. On the other hand, by \eqref{tist}, we have
$$
(F*G_\alpha)(x)=\left(J_a^\alpha f\right)(x),\quad \mbox{ a.e. } x\in [a,b],
$$
where $G_\alpha$ is given by \eqref{G-alpha}. Next, using   Lemma \ref{finj}, we have
$$
\frac{d(F*G_\alpha)}{dx}(x)=\left(G_\alpha*\frac{dF}{dx}\right)(x),\quad x\in \mathbb{R},
$$
which yields
$$
\frac{d}{dx} \left(J_a^\alpha f\right)(x)=\left(G_\alpha*\frac{dF}{dx}\right)(x),\quad \mbox{ a.e. } x\in [a,b],
$$
i.e.,
\begin{equation}\label{bs}
\left(\mathcal{D}_a^\alpha f\right)(x)=
\left(G_\alpha*\frac{dF}{dx}\right)(x),\quad \mbox{ a.e. } x\in [a,b].
\end{equation}
Moreover, using a similar argument as in the proof of Theorem \ref{T4}, we have
\begin{equation}\label{bs2}
\left(G_\alpha*\frac{dF}{dx}\right)(x)= \left(J_a^\alpha \frac{df}{dx}\right)(x),\quad \mbox{ a.e. } x\in [a,b].
\end{equation}
Combining \eqref{bs} with \eqref{bs2}, we obtain
\begin{equation}\label{bs3}
\left(\mathcal{D}_a^\alpha f\right)(x)=\left(J_a^\alpha \frac{df}{dx}\right)(x),\quad \mbox{ a.e. } x\in [a,b].
\end{equation}
Finally, since $\frac{df}{dx}\in L^1([a,b];\mathbb{R})$, using \eqref{bs3} and Theorem \ref{T4}, the desired result follows.
\end{proof}

Using a similar argument to that above, we obtain the following approximation result for $D_b^\alpha f$.

\begin{theorem}\label{TAD2}
Given $f\in V([a,b];\mathbb{R})$, $(a,b)\in \mathbb{R}^2$, $a<b$, we have
$$
\lim_{\alpha\to 0^+}\left\|\mathcal{D}_b^\alpha f -\frac{df}{dx}\right\|_{L^1([a,b];\mathbb{R})}=0.
$$
\end{theorem}

Next, we shall discuss the relation between the fractional integral operators of second kind and the fractional derivative operators introduced above.

\begin{theorem}\label{TINV1}
Let $f\in L^1([a,b];\mathbb{R}^2)$, $(a,b)\in \mathbb{R}^2$, $a<b$, and $\alpha>0$.  Then
$$
D_a^\alpha\left(S_a^\alpha f\right)(x)=f(x),\quad \mbox{ a.e. } x\in [a,b].
$$
\end{theorem}

\begin{proof}
Using Theorem \ref{THAYA}, we have
$$
J_a^\alpha\left(S_a^\alpha f\right)(x)=\int_a^x f(y)\,dy,\quad \mbox{ a.e. } x\in [a,b].
$$
Therefore, we obtain
$$
\frac{d}{dx}\left[J_a^\alpha\left(S_a^\alpha f\right)\right](x)=f(x),\quad \mbox{ a.e. } x\in [a,b],
$$
i.e.,
$$
D_a^\alpha\left(S_a^\alpha f\right)(x)=f(x),\quad \mbox{ a.e. } x\in [a,b].
$$
\end{proof}

\begin{theorem}\label{TINV2}
Let $f\in L^1([a,b];\mathbb{R}^2)$, $(a,b)\in \mathbb{R}^2$, $a<b$, and $\alpha>0$.  Then
$$
D_b^\alpha\left(S_b^\alpha f\right)(x)=-f(x),\quad \mbox{ a.e. } x\in [a,b].
$$
\end{theorem}

\begin{proof}
Using Theorem \ref{THAYA2}, we have
$$
J_b^\alpha\left(S_b^\alpha f\right)(x)=\int_x^b f(y)\,dy,\quad \mbox{ a.e. } x\in [a,b].
$$
Therefore, we obtain
$$
\frac{d}{dx}\left[J_b^\alpha\left(S_b^\alpha f\right)\right](x)=-f(x),\quad \mbox{ a.e. } x\in [a,b],
$$
i.e.,
$$
D_b^\alpha\left(S_b^\alpha f\right)(x)=-f(x),\quad \mbox{ a.e. } x\in [a,b].
$$
\end{proof}

Given $\alpha>0$ and $(a,b)\in \mathbb{R}^2$, $a<b$, we denote by $S_a^\alpha(L^p)$, $1\leq p\leq \infty$,  the functional space given by
$$
S_a^\alpha(L^p)=\left\{S_a^\alpha \varphi:\, \varphi\in L^p([a,b];\mathbb{R})\right\}.
$$
Similarly, we define the functional space $S_b^\alpha(L^p)$, $1\leq p\leq \infty$, by
$$
S_b^\alpha(L^p)=\left\{S_b^\alpha \varphi:\, \varphi\in L^p([a,b];\mathbb{R})\right\}.
$$

\begin{theorem}\label{TTK}
Let $f\in S_a^\alpha(L^1)$, where $\alpha>0$ and $(a,b)\in \mathbb{R}^2$, $a<b$. Then
$$
S_a^\alpha\left(D_a^\alpha f\right)(x)=f(x),\quad \mbox{ a.e. } x\in [a,b].
$$
\end{theorem}

\begin{proof}
By the definition of the functional space $S_a^\alpha(L^1)$, there exists a certain function $\varphi\in L^1([a,b];\mathbb{R})$ such that 
$$
f=S_a^\alpha \varphi.
$$
Using Theorem \ref{TINV1}, we obtain
$$
\left(D_a^\alpha f\right)(x)=\varphi(x),\quad \mbox{ a.e. } x\in [a,b].
$$
Therefore,
$$
S_a^\alpha\left(D_a^\alpha f\right)(x)=  \left(S_a^\alpha \varphi\right)(x)=f(x),\quad \mbox{ a.e. } x\in [a,b].
$$
\end{proof}

\begin{theorem}\label{TTK2}
Let $f\in S_b^\alpha(L^1)$, where $\alpha>0$ and $(a,b)\in \mathbb{R}^2$, $a<b$. Then
$$
S_b^\alpha\left(D_b^\alpha f\right)(x)=-f(x),\quad \mbox{ a.e. } x\in [a,b].
$$
\end{theorem}

\begin{proof}
By the definition of the functional space $S_b^\alpha(L^1)$, there exists a certain function $\varphi\in L^1([a,b];\mathbb{R})$ such that 
$$
f=S_b^\alpha \varphi.
$$
Using Theorem \ref{TINV2}, we obtain
$$
\left(D_b^\alpha f\right)(x)=-\varphi(x),\quad \mbox{ a.e. } x\in [a,b].
$$
Therefore,
$$
S_b^\alpha\left(D_b^\alpha f\right)(x)=  -\left(S_b^\alpha \varphi\right)(x)=-f(x),\quad \mbox{ a.e. } x\in [a,b].
$$
\end{proof}

It is important to note that   the results given by Theorem \ref{TTK} and Theorem \ref{TTK2} do not hold for an arbitrary 
function $f\in L^1([a,b];\mathbb{R})$. The following theorems
show this fact.

\begin{theorem}\label{KATR}
Let $f\in L^1([a,b];\mathbb{R})$, $(a,b)\in \mathbb{R}^2$, $a<b$, be a given function such that $J_a^\alpha f$, $\alpha>0$, has an absolutely continuous representative. Then
$$
S_a^\alpha\left(D_a^\alpha f\right)(x)=f(x)-\left(J_a^\alpha f\right)(a) \mathcal{S}\left(\frac{x-a}{\alpha}\right),\quad \mbox{ a.e. } x\in [a,b].
$$
\end{theorem}

\begin{proof}
Identifying $J_a^\alpha f$ to its absolutely continuous representative, we can write
\begin{equation}\label{mmb}
\left(J_a^\alpha f\right)(x)=\left(J_a^\alpha f\right)(a)+\int_a^x \varphi(y)\,dy,\quad a\leq x\leq b,
\end{equation}
where $\varphi\in L^1([a,b];\mathbb{R})$, which yields
\begin{equation}\label{yyh}
\left(D_a^\alpha f\right)(x)=\varphi(x),\quad \mbox{ a.e. } x\in [a,b].
\end{equation}
On the other hand, by Theorem \ref{THAYA}, we have 
\begin{equation}\label{mmnbv}
\int_a^x \varphi(y)\,dy =J_a^\alpha\left(S_a^\alpha \varphi\right)(x),\quad \mbox{ a.e. } x\in [a,b].
\end{equation}
Combining \eqref{mmb} with \eqref{mmnbv}, we obtain
$$
\left(J_a^\alpha f\right)(x)=\left(J_a^\alpha f\right)(a)+J_a^\alpha\left(S_a^\alpha \varphi\right)(x),\quad \mbox{ a.e. } x\in [a,b],
$$
i.e.,
$$
J_a^\alpha \left(f-S_a^\alpha \varphi\right)(x)=\left(J_a^\alpha f\right)(a),\quad \mbox{ a.e. } x\in [a,b].
$$
Applying the operator $S_a^\alpha$ to both sides of the above equation, and using again Theorem \ref{THAYA}, we obtain
$$
\int_a^x \left(f-S_a^\alpha \varphi\right)(y)\,dy=S_a^\alpha \left(\left(J_a^\alpha f\right)(a)\right)(x),\quad \mbox{ a.e. } x\in [a,b].
$$
Using \eqref{cas1S}, we obtain
$$
\int_a^x \left(f-S_a^\alpha \varphi\right)(y)\,dy=\alpha \left(J_a^\alpha f\right)(a) \int_0^\infty P\left(s,\frac{x-a}{\alpha}\right)\,ds, \quad \mbox{ a.e. } x\in [a,b].
$$
Taking the derivative with respect to $x$, and using \eqref{dig}, we obtain
\begin{equation}\label{jjkkh}
f(x)-\left(S_a^\alpha \varphi\right)(x)=\left(J_a^\alpha f\right)(a) \mathcal{S}\left(\frac{x-a}{\alpha}\right),\quad \mbox{ a.e. } x\in [a,b].
\end{equation}
Finally, combining \eqref{yyh} with \eqref{jjkkh}, the desired result follows.
\end{proof}

Using a similar argument to that above, we obtain the following result.

\begin{theorem}\label{KATR2}
Let $f\in L^1([a,b];\mathbb{R})$, $(a,b)\in \mathbb{R}^2$, $a<b$, be a given function such that $J_b^\alpha f$, $\alpha>0$, has an absolutely continuous representative. Then
$$
S_b^\alpha\left(D_b^\alpha f\right)(x)=-f(x)+\left(J_b^\alpha f\right)(b) \mathcal{S}\left(\frac{b-x}{\alpha}\right),\quad \mbox{ a.e. } x\in [a,b].
$$
\end{theorem}

The next results characterize the conditions for the existence of the fractional derivatives $D_a^\alpha$ and $D_b^\alpha$ in the space $AC([a,b];\mathbb{R})$.

\begin{theorem}\label{TYYR}
Let $\alpha>0$. If $f\in AC([a,b];\mathbb{R})$, $(a,b)\in \mathbb{R}^2$, $a<b$, then $D_a^\alpha f$ exists almost everywhere in $[a,b]$, and can be represented in the form
$$
\left(D_a^\alpha f\right)(x)=f(a) \left[\frac{1}{\alpha} E_1\left(\frac{x-a}{\alpha}\right)\right]+ \int_a^x \frac{1}{\alpha}E_1\left(\frac{x-y}{\alpha}\right)\frac{df}{dy}(y)\,dy
,\quad \mbox{ a.e. } x\in [a,b].
$$
\end{theorem}

\begin{proof}
Let $f\in AC([a,b];\mathbb{R})$. Then
$$
f(x)=f(a)+\int_a^x \frac{df}{dy}(y)\,dy,\quad a\leq x\leq b.
$$
Using Theorem \ref{THAYA}, we obtain
$$
f(x)=f(a)+S_a^\alpha\left(J_a^\alpha \frac{df}{dx}\right)(x),\quad \mbox{ a.e. } x\in [a,b],
$$
which yields
$$
\left(J_a^\alpha f\right)(x)=\left(J_a^\alpha f(a)\right)(x)+J_a^\alpha S_a^\alpha\left(J_a^\alpha \frac{df}{dx}\right)(x),\quad \mbox{ a.e. } x\in [a,b].
$$
Using \eqref{cas1} and Theorem \ref{THAYA}, we obtain
$$
\left(J_a^\alpha f\right)(x)=f(a) \left[\left(\frac{x-a}{\alpha} \right)E_1\left(\frac{x-a}{\alpha}\right)-e^{\frac{-(x-a)}{\alpha}}+1\right]+ \int_a^x \left(J_a^\alpha \frac{df}{dy}\right)(y)\,dy,
$$
for a.e. $x\in [a,b]$. Taking the derivative with respect to $x$, we obtain
$$
\left(D_a^\alpha f\right)(x)=\frac{f(a)}{\alpha} E_1\left(\frac{x-a}{\alpha}\right)+\left(J_a^\alpha \frac{df}{dx}\right)(x),\quad \mbox{ a.e. } x\in [a,b],
$$
i.e.,
$$
\left(D_a^\alpha f\right)(x)=f(a) \left[\frac{1}{\alpha} E_1\left(\frac{x-a}{\alpha}\right)\right]+ \int_a^x \frac{1}{\alpha}E_1\left(\frac{x-y}{\alpha}\right)\frac{df}{dy}(y)\,dy
,\quad \mbox{ a.e. } x\in [a,b].
$$
\end{proof}

Using a similar argument to that above, we obtain the following result for $D_b^\alpha f$.

\begin{theorem}\label{ghghg}
Let $\alpha>0$. If $f\in AC([a,b];\mathbb{R})$, $(a,b)\in \mathbb{R}^2$, $a<b$, then $D_b^\alpha f$ exists almost everywhere in $[a,b]$, and can be represented in the form
$$
\left(D_b^\alpha f\right)(x)=-f(b) \left[\frac{1}{\alpha} E_1\left(\frac{b-x}{\alpha}\right)\right]+ \int_x^b \frac{1}{\alpha}E_1\left(\frac{y-x}{\alpha}\right)\frac{df}{dy}(y)\,dy
,\quad \mbox{ a.e. } x\in [a,b].
$$
\end{theorem}

If $f\in V([a,b];\mathbb{R})$, $(a,b)\in \mathbb{R}^2$, $a<b$, using an approximate identity argument, we proved that (see Theorem \ref{TAD} and Theorem \ref{TAD2})  
\begin{equation}\label{sifou}
\lim_{\alpha\to 0^+}\left\|\mathcal{D}_a^\alpha f -\frac{df}{dx}\right\|_{L^1([a,b];\mathbb{R})}=0
\end{equation}
and
\begin{equation}\label{sifou2}
\lim_{\alpha\to 0^+}\left\|\mathcal{D}_b^\alpha f -\frac{df}{dx}\right\|_{L^1([a,b];\mathbb{R})}=0.
\end{equation}

Next, let  $f\in AC([a,b];\mathbb{R})$, $(a,b)\in \mathbb{R}^2$, $a<b$, be such that $f(a)=0$.  Using Theorem \ref{TYYR}, for $\alpha>0$,  we have
$$
\left(\mathcal{D}_a^\alpha f -\frac{df}{dx}\right)(x)=
\left(J_a^\alpha \frac{df}{dx}-\frac{df}{dx}\right)(x),\quad \mbox{ a.e. } x\in [a,b],
$$
which yields
$$
\left\|\mathcal{D}_a^\alpha f -\frac{df}{dx}\right\|_{L^1([a,b];\mathbb{R})}= \left\|J_a^\alpha \frac{df}{dx}-\frac{df}{dx}\right\|_{L^1([a,b];\mathbb{R})}.
$$
Passing to the limit as $\alpha\to 0^+$, and using Theorem \ref{T4}, we obtain \eqref{sifou}. Therefore, Theorem \ref{TAD} holds also true for every function $f\in AC([a,b];\mathbb{R})$ satisfying $f(a)=0$.

Similarly, let $f\in AC([a,b];\mathbb{R})$, $(a,b)\in \mathbb{R}^2$, $a<b$, be such that $f(b)=0$.  Using Theorem \ref{ghghg}, for $\alpha>0$,  we have
$$
\left(\mathcal{D}_b^\alpha f -\frac{df}{dx}\right)(x)=
\left(J_b^\alpha \frac{df}{dx}-\frac{df}{dx}\right)(x),\quad \mbox{ a.e. } x\in [a,b],
$$
which yields
$$
\left\|\mathcal{D}_b^\alpha f -\frac{df}{dx}\right\|_{L^1([a,b];\mathbb{R})}= \left\|J_b^\alpha \frac{df}{dx}-\frac{df}{dx}\right\|_{L^1([a,b];\mathbb{R})}.
$$
Passing to the limit as $\alpha\to 0^+$, and using Theorem \ref{T5}, we obtain \eqref{sifou2}. Therefore, Theorem \ref{TAD2} holds also true for every function $f\in AC([a,b];\mathbb{R})$ satisfying $f(b)=0$.

A formulation of the above facts is given by the following theorem.

\begin{theorem}\label{fmf}
Let $f\in AC([a,b];\mathbb{R})$, $(a,b)\in \mathbb{R}^2$, $a<b$, be such that $f(a)=0$. Then
$$
\lim_{\alpha\to 0^+}\left\|\mathcal{D}_a^\alpha f -\frac{df}{dx}\right\|_{L^1([a,b];\mathbb{R})}=0.
$$
Let $f\in AC([a,b];\mathbb{R})$, $(a,b)\in \mathbb{R}^2$, $a<b$, be such that $f(b)=0$. Then
$$
\lim_{\alpha\to 0^+}\left\|\mathcal{D}_b^\alpha f -\frac{df}{dx}\right\|_{L^1([a,b];\mathbb{R})}=0.
$$
\end{theorem}

We end this section with the following rule for fractional integration by parts.

\begin{theorem}\label{RIP}
Let $\alpha>0$ and $(a,b)\in \mathbb{R}^2$ be such that $a<b$.  If $f\in S_b^\alpha(L^\infty)$ and $g\in S_a^\alpha(L^1)$, then
$$
\int_a^b f(x) \left(D_a^\alpha g\right)(x)\,dx=-\int_a^b \left(D_b^\alpha f\right)(x)g(x)\,dx.
$$
\end{theorem}

\begin{proof}
Let $f\in S_b^\alpha(L^\infty)$ and $g\in S_a^\alpha(L^1)$. By the definition of the functional space $S_b^\alpha(L^\infty)$, there exists a certain function $\varphi_f\in L^\infty([a,b];\mathbb{R})$ such that 
$$
f(x)=\left(S_b^\alpha \varphi_f\right)(x),\quad \mbox{ a.e. } x\in [a,b].
$$
Similarly, by the definition of the functional space $S_a^\alpha(L^1)$, there exists a certain function $\varphi_g\in L^1([a,b];\mathbb{R})$ such that 
$$
g(x)=\left(S_a^\alpha \varphi_g\right)(x),\quad \mbox{ a.e. } x\in [a,b].
$$
Next, we have
$$
\int_a^b f(x) \left(D_a^\alpha g\right)(x)\,dx=\int_a^b \left(S_b^\alpha \varphi_f\right)(x) D_a^\alpha \left(S_a^\alpha \varphi_g\right)(x)\,dx.
$$
Using Theorem \ref{TINV1}, we obtain 
$$
\int_a^b f(x) \left(D_a^\alpha g\right)(x)\,dx=\int_a^b \left(S_b^\alpha \varphi_f\right)(x) \varphi_g(x)\,dx.
$$
Using Theorem \ref{TS3}, we obtain
\begin{equation}\label{hhgfz}
\int_a^b f(x) \left(D_a^\alpha g\right)(x)\,dx=\int_a^b g(x) \varphi_f(x)\,dx.
\end{equation}
On the other hand, by Theorem \ref{TINV2}, we have
\begin{equation}\label{cnhnn}
\left(D_b^\alpha f\right)(x)=D_b^\alpha\left(S_b^\alpha \varphi_f\right)(x)=-\varphi_f(x),\quad \mbox{ a.e. } x\in [a,b].
\end{equation}
Finally, combining \eqref{hhgfz} with \eqref{cnhnn}, the desired result follows.
\end{proof}

\section{A fractional relaxation equation}\label{sec5}

The standard form of a relaxation equation is given by
\begin{equation}\label{sre}
\frac{du(t)}{dt}+\lambda u(t)=f(t),
\end{equation}
where $\lambda$ is a positive constant and $f$ is a given function (see, for example \cite{D,M}). Equation \eqref{sre} models several physical phenomena, such as the Maxwell model, which describes the behavior of a viscoelastic material using a spring and a dashpot in series. In this case, $\lambda=\frac{E}{\eta}$, where $E$ is the elastic
modulus, $\eta$ is the viscosity coefficient, and $f(t)$ denotes $E$ multiplying the strain rate. In this section, we are concerned with the existence and uniqueness of solutions to the fractional relaxation equation 
\begin{equation}\label{FDES}
\begin{gathered}
\left(\mathcal{D}_0^\alpha u\right)(t)+\lambda u(t)= f(t,u(t)),\quad \mbox{ a.e. } t\in [0,1],\\
\left(J_0^\alpha u\right)(0)= 0,
\end{gathered}
\end{equation}
where $\alpha>0$,  $\lambda$ is a positive constant and $f: [0,1]\times \mathbb{R}\to \mathbb{R}$ is a given function. Moreover, we provide an iterative algorithm that converges to the solution.

\begin{definition}\label{desol}
We say that $u$ is a solution to \eqref{FDES} if $u\in C([0,1];\mathbb{R})$ is such that $J_0^\alpha u$ admits an absolutely continuous representative and $u$ satisfies \eqref{FDES}.
\end{definition}

Equation \eqref{FDES} is investigated under the following assumptions:
\begin{itemize}
\item[(A1)] $f: [0,1]\times \mathbb{R}\to \mathbb{R}$ is continuous.
\item[(A2)] For all $(t,x_1,x_2)\in [0,1]\times \mathbb{R}\times \mathbb{R}$, we have
$$
|f(t,x_1)-f(t,x_2)|\leq C_f |x_1-x_2|,
$$
where $C_f>0$ is a certain constant   satisfying
\begin{equation}\label{cttr}
\alpha\left(\lambda+C_f\right) \int_0^\infty P\left(s,\frac{1}{\alpha}\right)\,ds<1.
\end{equation}
\end{itemize}

The functional space $C([0,1];\mathbb{R})$ is equipped with the Chebyshev norm
$$
\|u\|_\infty=\max\{|u(t)|:\, 0\leq t\leq 1\},\quad u\in C([0,1];\mathbb{R}).
$$

Further, given $h\in C([0,1];\mathbb{R})$, we define 
\begin{eqnarray}\label{Th}
(Th)(t)=\left\{\begin{array}{lll}
\displaystyle \int_0^t \mathcal{S}\left(\frac{t-y}{\alpha}\right)h(y)\,dy, &\mbox{ if }& 0<t\leq 1,\\ \\
0 &\mbox{ if } & t=0.
\end{array}
\right.
\end{eqnarray}

\begin{lemma}\label{LLJJ}
We have
$$
T\left(C([0,1];\mathbb{R})\right)\subset C([0,1];\mathbb{R}).
$$
\end{lemma}

\begin{proof}
Let $h\in C([0,1];\mathbb{R})$ be fixed. We have
$$
|(Th)(t)|\leq \|h\|_{\infty} \int_0^t \mathcal{S}\left(\frac{t-y}{\alpha}\right)\,dy,\quad 0<t\leq 1,
$$  
i.e.,
$$
|(Th)(t)|\leq \alpha \|h\|_{\infty} \int_0^{\frac{t}{\alpha}} \mathcal{S}(z)\,dz,\quad 0<t\leq 1.
$$  
Passing to the limit as $t\to 0^+$, and using the dominated convergence theorem, we obtain
$$
\lim_{t\to 0^+} (Th)(t)=0,
$$
which proves that $Th$ is continuous at $0$. Further, let $r>0$ be small enough. We have
\begin{align}\label{tag}
\notag &\left|(Th)(1)-(Th)(1-r)\right|\\
\notag &= \left|\int_0^1 \mathcal{S}\left(\frac{1-y}{\alpha}\right)h(y)\,dy- \int_0^{1-r} \mathcal{S}\left(\frac{1-r-y}{\alpha}\right)h(y)\,dy\right|\\
\notag&= \alpha \left|\int_0^{\frac{1}{\alpha}} \mathcal{S}(z)h(1-\alpha z)\,dz- \int_0^{\frac{1-r}{\alpha}} \mathcal{S}(z)h(1-r-\alpha z)\,dz\right|\\
\notag&=\alpha \left| \int_0^{\frac{1-r}{\alpha}} \mathcal{S}(z)\left[h(1-\alpha z)-h(1-r-\alpha z)\right]\,dz +\int_{\frac{1-r}{\alpha}}^{\frac{1}{\alpha}} \mathcal{S}(z)h(1-\alpha z)\,dz\right|\\
\notag &\leq \alpha \left(\int_0^{\frac{1-r}{\alpha}} \mathcal{S}(z)\left|h(1-\alpha z)-h(1-r-\alpha z)\right|\,dz+\int_{\frac{1-r}{\alpha}}^{\frac{1}{\alpha}} \mathcal{S}(z)|h(1-\alpha z)|\,dz\right)\\
&:=\alpha \left(K_1(r)+K_2(r)\right).
\end{align}
On the other hand, we have
\begin{equation}\label{mods}
K_1(r)\leq \omega(r,h) \int_0^{\frac{1}{\alpha}} \mathcal{S}(z)\,dz,
\end{equation}
where  $\omega(\cdot,h)$ is the modulus of continuity of $h$, i.e.,
$$
\omega(r,h)=\sup\left\{|h(x)-h(y)|:\, (x,y)\in [0,1]\times[0,1],\, |x-y|\leq r\right\},\quad r>0.
$$
Note that since $h$ is uniformly continuous in $[0,1]$, we have
$$
\lim_{r\to 0^+} \omega(r,h)=0.
$$
Therefore, passing to the limit as $r\to 0^+$ in \eqref{mods}, we obtain
\begin{equation}\label{lim0}
\lim_{r\to 0^+} K_1(r)=0.
\end{equation}
Further, we have
$$
K_2(r)\leq \|h\|_\infty \int_{\frac{1-r}{\alpha}}^{\frac{1}{\alpha}} \mathcal{S}(z)\,dz.
$$ 
Passing to the limit as $r\to 0^+$, and using the dominated convergence theorem, we obtain
\begin{equation}\label{lim1}
\lim_{r\to 0^+} K_2(r)=0.
\end{equation}
Using \eqref{tag}, \eqref{lim0} and \eqref{lim1}, we obtain
$$
\lim_{r\to 0^+} \left|(Th)(1)-(Th)(1-r)\right|=0,
$$
which proves that $Th$ is continuous at $1$. Using a similar argument as above, we can show that $Th$ is continuous at any $t\in ]0,1[$. Therefore, we have
$$
Th\in C([0,1];\mathbb{R}).
$$
\end{proof}

Suppose now that $u\in C([0,1];\mathbb{R})$ is a solution to \eqref{FDES} in the sense of Definition \ref{desol}. Using Theorem \ref{KATR}, we obtain
$$
u(t)=S_0^\alpha\left(-\lambda u+f(\cdot,u(\cdot))\right)(t),\quad \mbox{ a.e. } t\in [0,1],
$$
i.e.,
$$
u(t)=T\left(-\lambda u+f(\cdot,u(\cdot))\right)(t),\quad \mbox{ a.e. } t\in [0,1].
$$
By continuity (see Lemma \ref{LLJJ}), we may identify both functions appearing in the above equation, so that we get
\begin{equation}\label{TUU}
u(t)=T\left(-\lambda u+f(\cdot,u(\cdot))\right)(t):=(Fu)(t),\quad t\in [0,1].
\end{equation}

Conversely, suppose that $u\in C([0,1]; \mathbb{R})$ is a solution to \eqref{TUU}. Then
$$
\left(J_0^\alpha u\right)(t)=J_0^\alpha\left[S_0^\alpha \left(-\lambda u+f(\cdot,u(\cdot))\right)\right](t),\quad \mbox{ a.e. } t\in [0,1].
$$
Using Theorem \ref{THAYA}, we obtain
$$
\left(J_0^\alpha u\right)(t)=\int_0^t \left[-\lambda u(s)+f(s,u(s))\right]\,ds,\quad \mbox{ a.e. } t\in [0,1].
$$
Observe that $J_0^\alpha u$ admits an absolutely continuous representative. Moreover,  by identification, we have
$$
\left(J_0^\alpha u\right)(0)=0
$$
and
$$
\frac{d}{dt}\left(J_0^\alpha u\right)(t)= -\lambda u(t)+f(t,u(t)),\quad \mbox{ a.e. } t\in [0,1],
$$
i.e.,
$$
\left(\mathcal{D}_0^\alpha u\right)(t)+\lambda u(t)=f(t,u(t)),\quad \mbox{ a.e. } t\in [0,1].
$$
Hence, $u$ is a solution to \eqref{FDES} in the sense of Definition \ref{desol}.

Therefore, we proved the following result.

\begin{lemma}\label{RESLT}
The following assertions are equivalent:
\begin{itemize}
\item[(i)] $u\in C([0,1];\mathbb{R})$ is a solution to  \eqref{FDES} in the sense of Definition \ref{desol}.
\item[(ii)] $u$ is a fixed point of the mapping
$$
F: C([0,1];\mathbb{R})\to C([0,1];\mathbb{R})
$$
given by \eqref{TUU}.
\end{itemize}
\end{lemma}

Next, we shall prove the following result.

\begin{lemma}\label{CNTM}
Under the assumptions (A1) and (A2), for all $(u,v)\in   C([0,1];\mathbb{R})\times C([0,1];\mathbb{R})$, we have
$$
\|Fu-Fv\|_\infty \leq  \alpha\left(\lambda+C_f\right) \left(\int_0^\infty P\left(s,\frac{1}{\alpha}\right)\,ds \right) \|u-v\|_\infty,
$$
where $F$ is the mapping given by \eqref{TUU}.
\end{lemma}

\begin{proof}
Let $(u,v)\in C([0,1];\mathbb{R})\times C([0,1];\mathbb{R})$.
Given $0<t\leq 1$, we have
$$
|(Fu)(t)-(Fv)(t)|=\left|T\left(-\lambda u+f(\cdot,u(\cdot))\right)(t)-T\left(-\lambda v+f(\cdot,v(\cdot))\right)(t)\right|,
$$
where $T$ is the mapping given by \eqref{Th}. Therefore, using the considered assumptions, we get
\begin{align*}
&|(Fu)(t)-(Fv)(t)|\\
&=\left|\int_0^t \mathcal{S}\left(\frac{t-y}{\alpha}\right)\left(-\lambda u(y)+f(y,u(y))\right)\,dy-\int_0^t \mathcal{S}\left(\frac{t-y}{\alpha}\right)\left(-\lambda v(y)+f(y,v(y))\right)\,dy\right|\\
&\leq \int_0^t \mathcal{S}\left(\frac{t-y}{\alpha}\right) \big[ \lambda |u(y)-v(y)| +|f(y,u(y))-f(y,v(y))|\big]\,dy\\
&\leq  \left(\lambda +C_f\right)\left(\int_0^t \mathcal{S}\left(\frac{t-y}{\alpha}\right) \,dy\right)\|u-v\|_\infty\\
&=\alpha \left(\lambda +C_f\right)\left(\int_0^{\frac{t}{\alpha}} \mathcal{S}(z) \,dz\right)\|u-v\|_\infty\\
&\leq \alpha \left(\lambda +C_f\right)\left(\int_0^{\frac{1}{\alpha}} \mathcal{S}(z) \,dz\right)\|u-v\|_\infty\\
&=\alpha\left(\lambda+C_f\right) \left(\int_0^\infty P\left(s,\frac{1}{\alpha}\right)\,ds \right) \|u-v\|_\infty,
\end{align*}
which proves the desired result.
\end{proof}

Now, we are able to state and prove the following existence result.

\begin{theorem}\label{EXISTENCE}
Under the assumptions (A1) and (A2), Problem \eqref{FDES} admits one and only one solution $u^*\in C([0,1];
\mathbb{R})$ in the sense of Definition \ref{desol}.  Moreover, for every $u_0\in C([0,1];\mathbb{R})$, the Picard sequence $\{u_n\}\subset C([0,1];\mathbb{R})$, given by
\begin{eqnarray*}
u_{n+1}(t)=\left\{\begin{array}{lll}
\displaystyle \int_0^t \mathcal{S}\left(\frac{t-y}{\alpha}\right)\big[-\lambda u_n(y)+f(y,u_n(y))\big]\,dy, &\mbox{ if }& 0<t\leq 1,\\ \\
0 &\mbox{ if } & t=0
\end{array}
\right.;\quad n\in \mathbb{N},
\end{eqnarray*}
converges uniformly to $u^*$.
\end{theorem}

\begin{proof}
From Lemma \ref{CNTM} and \eqref{cttr}, the mapping $F: C([0,1];\mathbb{R})\to C([0,1];\mathbb{R})$ is a contraction. Therefore, by Banach contraction principle, the mapping $F$ has one and only one solution $u^*\in C([0,1];\mathbb{R})$. Moreover, for every $u_0\in C([0,1];\mathbb{R})$, the Picard sequence $\{u_n\}\subset C([0,1];\mathbb{R})$, given by
$$
u_{n+1}=Fu_n,\quad n\in \mathbb{N},
$$
converges to $u^*$ with respect to the norm $\|\cdot\|_\infty$. Finally, by Lemma \ref{RESLT}, the desired result follows.
\end{proof}

\vspace{1cm}

\noindent Mohamed Jleli\\
Department of Mathematics, College of Science, King Saud University, P.O. Box 2455, Riyadh, 11451, Saudi Arabia\\
E-mail: jleli@ksu.edu.sa\\

\noindent Bessem Samet\\
Department of Mathematics, College of Science, King Saud University, P.O. Box 2455, Riyadh, 11451, Saudi Arabia\\
E-mail: bsamet@ksu.edu.sa


\begin{thebibliography}{99}


\bibitem{AS}
M. Abramowitz, I.A. Stegun, Handbook of Mathematical Functions with Formulas, Graphs, and Mathematical Tables, Natl. Bur. Stand., Appl. Math. Ser., vol. 55, U.S. Government Printing Office, Washington, DC, USA, 1964.



\bibitem{AG}
O.P. Agrawal, Some generalized fractional calculus operators and their applications in integral equations, Frac. Cal. Appl. Anal. 15 (4) (2012) 700--711.

\bibitem{AL}
R. Almeida, A Caputo fractional derivative of a function with respect to another function, Commun. Nonlinear Sci. Numer. Simulat. 44 (2017) 460--481.




\bibitem{BA}
R.L. Bagley, P.J. Torvik, On the fractional calculus model of viscoelastic behavior, Journal of Rheology. 30 (1) 133--155 (1986).

\bibitem{CH}
D.C. Champerey, A Handbook of Fourier Theorems,
Cambridge University Press, Cambridge, New York (1987).

\bibitem{CO}
S. Corlay, J. Lebovits, J.L. V\'ehel, Multifractional stochastic volatility models, Math. Finance. 24 (2) (2014) 364--402.

\bibitem{D}
K. Diethelm, The analysis of fractional differential equations, Springer, Heidelberg, 2010.

\bibitem{FA}
C.Q. Fang, H.Y. Sun, J.P. Gu, Application of fractional calculus methods to viscoelastic response of amorphous shape memory polymers, J. Mech. 31 (4) (2015) 427--432.


\bibitem{GA}
R. Garra, R. Gorenflo, F. Polito, \v{Z}. Tomovski,
Hilfer-Prabhakar derivatives and some applications,
Appl Math and Comput. 242 (2014) 576--589.






\bibitem{G}
M. Geller, E.W. Ng, A table of integrals of the exponential integral, Journal of Research the National Bureau af Standards-B. Mathematics and Mathematical Science. 73B (1969)
191--210.




\bibitem{GH}
P. Ghamisi, M.S. Couceiro, J.A. Benediktsson, N.M.F. Ferreira, An efficient method for segmentation of images based on fractional calculus and natural selection,
Expert Syst. Appl. 39 (16) (2012) 12407--12417.


\bibitem{HI}
R. Hilfer,  Applications of fractional calculus in physics, World Scientific, Singapore (2000).



\bibitem{JKS}
M. Jleli, M. Kirane, B. Samet, A numerical approach based on ln-shifted Legendre polynomials for solving a fractional model of pollution , Math. Method. Appl. Sci. 40 (18) (2017) 7356--7367.




\bibitem{KST}
A.A. Kilbas, H.M. Srivastava, J.J. Trujillo, Theory and Applications of the Fractional Differential Equations, volume 204. Elsevier, Amsterdam, 2006.


\bibitem{KIR}
V. Kiryakova, Generalized Fractional Calculus and Applications, Longman, Harlow; Wiley, New York, 1994.


\bibitem{KU}
V.V. Kulish, J.L. Lage, Application of fractional calculus to fluid mechanics, J. Fluids Eng. 124 (3) (2002) 803--806.

\bibitem{M}
F. Mainardi, Fractional relaxation-oscillation and fractional diffusion-wave phenomena, Chaos, Solitons and Fractals. 
7 (9) (1996) 1461--1477.

\bibitem{SKM}
S.G. Samko, A.A. Kilbas,  O.I. Marichev, Fractional Integrals and Derivatives: Theory
and Applications. Gordon and Breach, Yverdon, 1993.


\bibitem{SR}
H. Srivastava, \v{Z}. Tomovski,
Fractional calculus with an integral operator containing a generalized Mittag-Leffler function in the kernel,
Appl Math and Comput. 211 (1) (2009) 198--210.



\end{thebibliography}
\end{document}